\documentclass{amsart}

\usepackage[all]{xy}

\usepackage{latexsym}

\usepackage{amssymb}
\usepackage{amsfonts}
\usepackage{amscd}
\usepackage{amsmath,amsthm}

\newtheorem{lemma}{Lemma}[section]

\newtheorem{theorem}[lemma]{Theorem}
\newtheorem{corollary}[lemma]{Corollary}

\newtheorem{prop}[lemma]{Proposition}
\newtheorem{thm}[lemma]{Theorem}
\newtheorem{cor}[lemma]{Corollary}

{

}

\theoremstyle{definition}

\newtheorem{question}[lemma]{Question}

\theoremstyle{remark}

\numberwithin{equation}{section}

\begin{document}

\title{Noncommutative Tsen's theorem in dimension one}
\author{A. Nyman}
\address{Department of Mathematics, 516 High St, Western Washington University, Bellingham, WA 98225-9063}
\email{adam.nyman@wwu.edu}
\keywords{}
\thanks{2010 {\it Mathematics Subject Classification. } Primary 14A22, 14H45; Secondary 16S38}

\begin{abstract}
Let $k$ be a field.  In this paper, we find necessary and sufficient conditions for a noncommutative curve of genus zero over $k$ to be a noncommutative $\mathbb{P}^{1}$-bundle.  This result can be considered a noncommutative, one-dimensional version of Tsen's theorem.  By specializing this theorem, we show that every arithmetic noncommutative projective line is a noncommutative curve, and conversely we characterize exactly those noncommutative curves of genus zero which are arithmetic.  We then use this characterization, together with results from \cite{nyman}, to address some problems posed in \cite{kussin}.
\end{abstract}

\maketitle

\pagenumbering{arabic}

\section{Introduction}
Throughout this paper, $k$ will denote a field.  In \cite{kussin}, the concept of a noncommutative curve of genus zero is defined as a small $k$-linear abelian category $\sf H$ such that
\begin{itemize}
\item{} each object of ${\sf H}$ is noetherian,

\item{} all morphism and extension spaces in ${\sf H}$ are finite dimensional over $k$,

\item{} ${\sf H}$ admits an Auslander-Reiten translation, i.e. an autoequivalence $\tau$ such that Serre duality $\operatorname{Ext}_{{\sf H}}^{1}(\mathcal{E},\mathcal{F}) \cong D\operatorname{Hom}_{{\sf H}}(\mathcal{F},\tau \mathcal{E})$ holds, where $D(-)$ denotes the $k$-dual,

\item{} ${\sf H}$ has an object of infinite length, and

\item{} ${\sf H}$ has a tilting object.
\end{itemize}
One motivation for the definition is that if $C$ is a smooth projective curve of genus zero over $k$, then the category of coherent sheaves over $C$ satisfies these properties.  Furthermore, Kussin calls the category ${\sf H}$ {\it homogeneous} if
\begin{itemize}
\item{} for all simple objects $\mathcal{S}$ in ${\sf H}$, $\operatorname{Ext}_{{\sf H}}^{1}(\mathcal{S},\mathcal{S}) \neq 0$.
\end{itemize}
If ${\sf H}$ is not homogenous (e.g. if ${\sf H}$ is a weighted projective line) then ${\sf H}$ is birationally equivalent to a homogeneous noncommutative curve of genus zero \cite[p. 2]{kussin}.  Therefore, from the perspective of noncommutative birational geometry, the homogeneous curves play a crucial role.

If ${\sf H}$ is a homogenous noncommutative curve of genus zero and $\mathcal{L}$ is a line bundle on ${\sf H}$, then there exists an indecomposable bundle $\overline{\mathcal{L}}$ and an irreducible morphism $\mathcal{L} \longrightarrow \overline{\mathcal{L}}$ coming from an AR sequence starting at $\mathcal{L}$.  Kussin calls the bimodule $M:={}_{\operatorname{End }(\overline{\mathcal{L}})}\operatorname{Hom}_{{\sf H}}(\mathcal{L},\overline{\mathcal{L}})_{\operatorname{End }(\mathcal{L})}$ the {\it underlying bimodule} of ${\sf H}$.  It turns out that the only possibilities for the left-right dimensions of $M$ are $(1,4)$ and $(2,2)$.

On the other hand, in \cite{vandenbergh}, M. van den Bergh introduces the notion of a noncommutative $\mathbb{P}^{1}$-bundle over a pair of commutative schemes $X$, $Y$.  In particular, if $K$ and $L$ are finite extensions of $k$ and $N$ is a $k$-central $K-L$-bimodule of finite dimension as both a $K$-module and an $L$-module, then one can form the $\mathbb{Z}$-algebra $\mathbb{S}^{n.c.}(N)$, the noncommutative symmetric algebra of $N$ (see Section \ref{section.ncsym} for details).  The {\it noncommutative $\mathbb{P}^{1}$-bundle generated by $N$}, $\mathbb{P}^{n.c.}(N)$, is defined to be the quotient of the category of graded right $\mathbb{S}^{n.c.}(N)$-modules modulo the full subcategory of direct limits of right bounded modules.  It is natural to ask whether a homogeneous noncommutative curve of genus zero is a noncommutative $\mathbb{P}^{1}$-bundle generated by $M$, at least under the necessary condition that $\operatorname{End}(\mathcal{L})$ and $\operatorname{End}(\overline{\mathcal{L}})$ are commutative.  Our main result is that this is the case. Before we state it precisely, we need to introduce some notation.  If ${\sf C}$ is a noetherian category, then there exists a unique locally noetherian category $\tilde{\sf C}$ whose full subcategory of noetherian objects is ${\sf C}$ \cite[Theorem 2.4]{blowup}.  Furthermore, if ${\sf C}$ and ${\sf D}$ are $k$-linear categories and there exists a $k$-linear equivalence ${\sf C} \longrightarrow {\sf D}$, we write ${\sf C} \equiv {\sf D}$.

Our main theorem is the following (Theorem \ref{thm.almostring}):
\begin{theorem} \label{thm.main}
If ${\sf H}$ is a homogeneous noncommutative curve of genus zero over with underlying bimodule $M$ such that $\operatorname{End}(\mathcal{L})$ and $\operatorname{End}(\overline{\mathcal{L}})$ are commutative, then
$$
\tilde{{\sf H}} \equiv \mathbb{P}^{n.c.}(M).
$$
Conversely, if $K$ and $L$ are finite extensions of $k$ and $N$ is a $k$-central $K-L$-bimodule of left-right dimension $(2,2)$ or $(1,4)$, then $\mathbb{P}^{n.c.}(N)$ is a noncommutative curve of genus zero with underlying bimodule $N$.
\end{theorem}
If we think of a homogeneous noncommutative curve of genus zero as being a kind of conic bundle over the pair $\operatorname{Spec }(\operatorname{End}(\mathcal{L}))$, $\operatorname{Spec }(\operatorname{End}(\overline{\mathcal{L}}))$, then Theorem \ref{thm.main} can be interpreted as saying that all noncommutative one-dimensional conic bundles are noncommutative one-dimensional $\mathbb{P}^{1}$-bundles, whence the connection to Tsen's theorem.

The main idea behind the proof of Theorem \ref{thm.main} is that one can form a $\mathbb{Z}$-algebra coordinate ring for $\tilde{{\sf H}}$ by employing {\it all} of the indecomposable bundles in ${\sf H}$, as opposed to constructing the $\mathbb{Z}$-graded algebra using only the $\tau$-orbit of $\mathcal{L}$, as is done in \cite{kussin}.  One then shows that this $\mathbb{Z}$-algebra is isomorphic to the noncommutative symmetric algebra $\mathbb{S}^{n.c.}(M)$.  Fortunately, many of the technical details needed to construct this isomorphism are provided by \cite[Proposition 2.1 and 2.2]{dlab} and \cite{kussin}.

\subsection{Noncommutative curves of genus zero and arithmetic noncommutative projective lines}

We now describe a specialization of Theorem \ref{thm.main} and its applications.  In \cite{nyman}, the notion of arithmetic noncommutative projective line over $k$ is studied.  These are none other than spaces of the form $\mathbb{P}^{n.c.}(V)$ where $V$ is a $k$-central, $K-K$-bimodule of left-right dimension $(2,2)$.  Among other things, the automorphism groups of these spaces are computed in terms of arithmetic data defining $V$ \cite[Lemma 7.2, Lemma 7.3, and Theorem 7.4]{nyman}, and isomorphism invariants are determined \cite[Theorem 5.5]{nyman}.  As a consequence of Theorem \ref{thm.main} we have the following (Corollary \ref{newcor.main}):

\begin{cor} \label{cor.main}
Suppose ${\sf H}$ is a homogeneous noncommutative curve of genus zero such that $M$ has left-right dimension $(2,2)$, and $\operatorname{End }(\mathcal{L})$ and $\operatorname{End }(\overline{\mathcal{L}})$ are isomorphic and commutative.  Then $\tilde{{\sf H}} \equiv \mathbb{P}^{n.c.}(M)$, so that $\tilde{{\sf H}}$ is an arithmetic noncommutative projective line.

Conversely, every arithmetic noncommutative projective line $\mathbb{P}^{n.c.}(V)$ is a homogeneous noncommutative curve of genus zero with underlying bimodule $V$.
\end{cor}

Corollary \ref{cor.main} has several applications to both arithmetic noncommutative projective lines and homogeneous noncommutative curves of genus zero, which we now describe.

\subsection{Applications to arithmetic noncommutative projective lines}

In \cite[Theorem 3.7]{nyman}, homological techniques are employed to give a proof that the noncommutative symmetric algebra of a $k$-central $K-K$-bimodule of left-right dimension $(2,2)$, $V$, is a domain in the sense that if $x \in \mathbb{S}^{n.c.}(V)_{ij}$ and $y \in \mathbb{S}^{n.c.}(V)_{jl}$ then $xy=0$ implies that $x=0$ or $y=0$.  Using Corollary \ref{cor.main}, we produce a much shorter proof of a generalization (Proposition \ref{prop.domain}).

In \cite{nyman}, a classification of spaces of the form $\mathbb{P}^{n.c.}(V)$ up to $k$-linear equivalence, and a classification of isomorphisms between such spaces, is described. It is natural, then, to determine whether a Bondal-Orlov theorem exists for arithmetic noncommutative projective lines.  It is known that such a theorem holds more generally for homogeneous noncommutative curves of genus zero, and therefore, by Corollary \ref{cor.main}, we deduce (Theorem \ref{thm.bondalorlov}) a Bondal-Orlov theorem for arithmetic noncommutative projective lines.  Since we could not find any proofs of this fact in the literature, we provide a proof for the readers convenience.

\subsection{Applications to homogeneous noncommutative curves of genus zero} \label{section.questions}
We use Corollary \ref{cor.main} to address three questions that appear in \cite{kussin}, which we now describe.  In \cite[1.1.5]{kussin}, Kussin partitions the collection of underlying bimodules into three classes, or ``orbit cases".  Orbit case I consists of those $M$ which have left-right dimension $(1,4)$.  Orbit case II consists of those $M$ with left-right dimension $(2,2)$ such that the group of point preserving $k$-linear autoequivalences acts transitively on the set of line bundles.  Orbit case III consists of all other $M$ of left-right dimension $(2,2)$.  Kussin poses a number of problems related to this trichotomy \cite{kussin}.  For example he asks the following

\begin{question} \label{question.first} (\cite[Problem 1.1.11]{kussin}):
Is there a criterion from which one can easily decide whether a given bimodule of left-right dimension $(2,2)$ is in orbit case II or III?
\end{question}
We now describe the next question we address.  In certain cases in which $M$ is a non-simple bimodule of left-right dimension $(2,2)$, and the center of $\operatorname{End}(\overline{\mathcal{L}})$ has cyclic Galois group over $k$, Kussin obtains a precise description of the automorphism group of ${\sf H}$, denoted $\operatorname{Aut }(\mathbb{X})$ \cite[Theorem 5.3.4]{kussin}.  By definition, this is the group of isomorphism classes of $k$-linear equivalences fixing $\mathcal{L}$ up to isomorphism.  This leads to the following

\begin{question} \label{question.third} (\cite[Problem 5.4.7]{kussin}):
What is $\operatorname{Aut }(\mathbb{X})$ in other situations, and in particular in case $M$ is a simple bimodule?
\end{question}
Another question is the following
\begin{question} \label{question.second} (\cite[third part of Problem 5.4.6]{kussin}): Find a general functorial formula for the Serre functor.
\end{question}

Corollary \ref{cor.main}, together with results from \cite{nyman}, allow us to address Questions \ref{question.first}, \ref{question.third} and \ref{question.second} in the case that an underlying bimodule in ${\sf H}$ has left-right dimension $(2,2)$, $\operatorname{End}(\mathcal{L})$ and $\operatorname{End}(\overline{\mathcal{L}})$ are commutative and isomorphic, and $k$ is perfect with $\operatorname{char }k \neq 2$.  Although we are not able to provide a complete answer to Question \ref{question.first}, we are able to address a closely related problem (Proposition \ref{prop.first}).  To state the result, we need to introduce some notation.  If $\sigma$ is a $k$-linear automorphism of $K$, then we let $K_{\sigma}$ denote the $k$-central $K-K$-bimodule whose underlying set is $K$ and whose bimodule structure is given by $a \cdot x \cdot b := ax\sigma(b)$.

\begin{prop}
Suppose $k$ is perfect and $\operatorname{char }k \neq 2$.  If ${\sf H}$ is a homogeneous noncommutative curve of genus zero such that an underlying bimodule, $M$, has left-right dimension $(2,2)$, and $\operatorname{End}(\mathcal{L})$ and $\operatorname{End}(\overline{\mathcal{L}})$ are commutative and isomorphic, then the group of $k$-linear autoequivalences of ${\sf H}$ acts transitively on line bundles if and only if there exists $\sigma, \epsilon \in \operatorname{Gal}(K/k)$ such that $M \cong K_{\sigma} \otimes {}^{*}M \otimes K_{\epsilon}$.
\end{prop}

{\it Acknowledgements:}  I thank D. Kussin for patiently answering my many questions regarding \cite{kussin}.  I thank D. Chan for suggesting I should generalize a previous version of Theorem \ref{thm.main} to include $(1,4)$-bimodules, for pointing out an error in an earlier version of the proof of Proposition \ref{prop.newdimcount}, for helping me correct it, and for pointing out the relationship between Theorem \ref{thm.main} and Tsen's theorem.

\section{Noncommutative symmetric algebras} \label{section.ncsym}
The purpose of this section is to recall the definition of the noncommutative symmetric algebra of a bimodule from \cite{vandenbergh}. For the remainder of the paper, we let $K$ and $L$ denote finite field extensions of $k$.  We let $N$ denote a $k$-central $K-L$-bimodule which has finite dimension as both a $K$-module and an $L$-module.  We denote the restriction of scalars of $N$ to $K \otimes_{k} 1$ (resp. $1\otimes_{k} L$) by ${}_{K}N$ (resp. $N_{L}$).  If $\operatorname{dim}_{K}({}_{K}N)=m$ and $\operatorname{dim}_{L}(N_{L})=n$, we say $N$ is an $(m,n)$-{\it bimodule}.  For the remainder of the paper, all bimodules will be $k$-central with $m,n < \infty$.  We will use the fact that if $K'$ is another finite extension of $k$, $M$ is a $K-L$-bimodule with left basis $\{ v_{1},\ldots,v_{m} \}$ and $N$ is an $L-K'$-bimodule with left basis $\{ w_{1}, \ldots, w_{n} \}$, then $\{v_{i}\otimes w_{j} | 1 \leq i \leq m \mbox{ and } 1 \leq j \leq n \}$ is a left basis for $M \otimes_{L} N$.

We need to recall (from \cite{vandenbergh}) the notion of left and right dual of $N$.  The {\it right dual of $N$}, denoted $N^{*}$, is the set
$\operatorname{Hom}_{L}(N_{L},L)$ with action
$$
(a \cdot \psi \cdot b)(n)=a\psi(bn)
$$
for all $\psi \in \operatorname{Hom}_{L}(N_{L},L)$, $a\in L$ and $b \in K$.  We note that
$N^{*}$ is a $L \otimes_{k}K$-module.

The {\it left dual of $N$}, denoted ${}^{*}N$, is the set
$\operatorname{Hom}_{K}({}_{K}N,K)$ with action
$$
(a \cdot \phi
\cdot b)(n)=b \phi(na)
$$
for all $\phi \in \operatorname{Hom}_{K}({}_{K}N,K)$, $a \in L$ and $b \in K$.  As above,
${}^{*}N$ is a $L \otimes_{k}K$-module.  This assignment extends to morphisms between $k$-central $K-L$-bimodules in the obvious way.

We set
$$
N^{i*}:=
\begin{cases}
N & \text{if $i=0$}, \\
(N^{i-1*})^{*} & \text{ if $i>0$}, \\
{}^{*}(N^{i+1*}) & \text{ if $i<0$}.
\end{cases}
$$

As in \cite[Proposition 3.7]{hart}, for each $i$, both pairs of functors
$$
(-\otimes N^{i*},-\otimes N^{i+1*})
$$
and
$$
(-\otimes {}^{*}(N^{i+1*}), -\otimes N^{i+1*})
$$
between the category of $K$-modules and the category of $L$-modules, have adjoint structures.  The Eilenberg-Watts theorem implies that the units of these adjoint pairs induce maps of bimodules, and to describe them we introduce some notation.  We let
$$
F_{i} = \begin{cases}
K & \mbox{if $i$ is even, and} \\
L & \mbox{if $i$ is odd.}
\end{cases}
$$
We suppose $\{\phi_{1},\ldots,\phi_{n}\}$ is a right basis for $N^{i*}$ and $\{f_{1},\ldots, f_{n}\}$ is the corresponding dual left basis for $N^{i+1*}$.  We suppose $\{\phi'_{1},\ldots,\phi'_{n}\}$ is the right basis for ${}^{*}(N^{i+1*})$ dual to $\{f_{1},\ldots, f_{n}\}$.

The unit of the first adjoint pair induces $\eta_{i,N}: F_{i} \longrightarrow N^{i*} \otimes  N^{i+1*}$, defined by
$$
\eta_{i,N}(a) = a \sum_{j} \phi_{j} \otimes f_{j}
$$
while the unit of the second adjoint pair induces $\eta_{i,N}':F_{i} \longrightarrow {}^{*}(N^{i+1*}) \otimes N^{i+1*}$ defined by
$$
\eta_{i,N}'(a) = a \sum_{j} \phi'_{j} \otimes f_{j}.
$$
We denote the image of $\eta_{i,N}$ by $Q_{i,N}$ and the image of $\eta_{i,N}'$ by $Q_{i,N}'$.  The subscript $N$ will be dropped when there is no chance for confusion.

The next result follows immediately from \cite[Lemma 6.6]{duality}.
\begin{lemma} \label{lemma.eta}
Let $N$ and $P$ be $k$-central $K-L$-bimodules.
\begin{enumerate}
\item{} Given an isomorphism $\Gamma: N^{*} \longrightarrow P^{*}$, there exists a unique isomorphism $\Delta: N \longrightarrow {}^{*}(P^{*})$ such that $(\Delta \otimes \Gamma) (Q_{0,N})=Q'_{0,P}$.

\item{} Given an isomorphism $\Delta: {}^{*}N \longrightarrow {}^{*}P$ (resp. $\Gamma: N \longrightarrow P$), there exists a unique isomorphism $\Gamma: N \longrightarrow P$ (resp. $\Delta:{}^{*}N \longrightarrow {}^{*}P$) such that $(\Delta \otimes \Gamma) (Q_{-1,N})=Q_{-1,P}$.
\end{enumerate}
\end{lemma}
We now introduce a convention that will be in effect throughout this paper:  all unadorned tensor products will be bimodule tensor products over the appropriate base ring.

We next recall the definition of $\mathbb{Z}$-algebra from \cite[Section 2]{quadrics}: a $\mathbb{Z}$-{\it algebra} is a ring $A$ with decomposition $A=\oplus_{i,j \in \mathbb{Z}}A_{ij}$ into $k$-vector spaces, such that multiplication has the property $A_{ij}A_{jk} \subset A_{ik}$ while $A_{ij}A_{kl}=0$ if $j \neq k$.  Furthermore, for $i \in \mathbb{Z}$, there is a local unit $e_{i} \in A_{ii}$, such that if $a \in A_{ij}$, then $e_{i}a=a=ae_{j}$.

Now we are finally ready to recall (from \cite{vandenbergh}) the definition of the noncommutative symmetric algebra of $N$.  All tensor products will be defined either over $K$ or $L$, and the choice will be clear from context.  The {\it noncommutative symmetric algebra generated by $N$}, denoted $\mathbb{S}^{n.c.}(N)$, is the $\mathbb{Z}$-algebra $\underset{i,j \in
\mathbb{Z}}{\oplus}A_{ij}$ with components defined as follows:
\begin{itemize}
\item{} $A_{ij}=0$ if $i>j$.

\item{} $A_{ii}=K$ for $i$ even,

\item{} $A_{ii}=L$ for $i$ odd, and

\item{} $A_{i i+1}=N^{i*}$.
\end{itemize}
In order to define $A_{ij}$ for $j>i+1$, we introduce some notation: we define $T_{i i+1} := A_{i i+1}$, and, for $j>i+1$, we define
$$
T_{ij} := A_{i i+1} \otimes A_{i+1 i+2} \otimes \cdots \otimes A_{j-1 j}.
$$
We let $R_{i i+1}:= 0$, $R_{i i+2}:=Q_{i}$,
$$
R_{i i+3}:=Q_{i} \otimes N^{(i+2)*}+N^{i*} \otimes Q_{i+1},
$$
and, for $j>i+3$, we let
$$
R_{ij} := Q_{i} \otimes T_{i+2 j}+T_{i i+1}\otimes Q_{i+1} \otimes T_{i+3 j}+\cdots + T_{i j-2} \otimes Q_{j-2}.
$$

\begin{itemize}
\item{} For $j>i+1$, we define $A_{ij}$ as the quotient $T_{ij}/R_{ij}$.
\end{itemize}
Multiplication in $\mathbb{S}^{n.c.}(N)$ is defined as follows:
\begin{itemize}

\item{} if $x \in A_{ij}$, $y \in A_{lk}$ and $j \neq l$, then $xy=0$,

\item{} if $x \in A_{ij}$ and $y \in A_{jk}$, with either $i=j$ or $j=k$, then $xy$ is induced by the usual scalar action,

\item{}  otherwise, if $i<j<k$, we have
\begin{eqnarray*}
A_{ij} \otimes A_{jk} & = & \frac{T_{ij}}{R_{ij}} \otimes \frac{T_{jk}}{R_{jk}}\\
& \cong & \frac{T_{ik}}{R_{ij}\otimes T_{jk}+T_{ij} \otimes
R_{jk}}.
\end{eqnarray*}
Since $R_{ij} \otimes T_{jk}+T_{ij} \otimes R_{jk}$ is a submodule of $R_{ik}$,
there is thus an epi $\mu_{ijk}:A_{ij} \otimes A_{jk} \longrightarrow
A_{ik}$.
\end{itemize}

\section{Proof of the Main Theorem}
Our goal in this section is to prove Theorem \ref{thm.main}.  We begin the section with a description of the notation we will utilize as well as a statement of our assumptions.  Throughout the rest of this paper, ${\sf H}$ will denote a homogeneous noncommutative curve of genus zero, and all unadorned $\operatorname{Hom}$'s will be over ${\sf H}$.

A {\it bundle} in ${\sf H}$ is an object that doesn't have a simple subobject.  A {\it line bundle} in ${\sf H}$ is bundle of rank one (see \cite[p.136]{lenzingreiten} for the definition of rank).
\begin{itemize}
\item{} We let $\mathcal{L}$ denote a line bundle in ${\sf H}$. There is an indecomposable bundle $\overline{\mathcal{L}}$ and an irreducible morphism $\mathcal{L} \longrightarrow \overline{\mathcal{L}}$ coming from the AR sequence starting at $\mathcal{L}$ \cite[1.1.2]{kussin}.

\item{} We let $\tau^{-1}$ denote a fixed quasi-inverse of the Auslander-Reiten translation, $\tau$, of ${\sf H}$.

\item{} We let $M$ denote the $\operatorname{End }(\overline{\mathcal{L}})-\operatorname{End}(\mathcal{L})$-bimodule $\operatorname{Hom}(\mathcal{L},\overline{\mathcal{L}})$.

\item{} We assume that $\operatorname{End}(\overline{\mathcal{L}})$ and $\operatorname{End}(\mathcal{L})$ are commutative.
\end{itemize}
We recall that $\operatorname{End}(\mathcal{L})$ and $\operatorname{End }(\overline{\mathcal{L}})$ are automatically division rings of finite dimension over $k$ by \cite[Lemma 1.3]{lenzingreiten}.  Therefore, by the above assumption, they are finite extension fields of $k$.

If $\mathcal{N}$ is an indecomposable bundle, we let
\begin{equation} \label{eqn.arsequence}
0 \longrightarrow \mathcal{N} \overset{h}{\longrightarrow} \mathcal{E} \overset{p}{\longrightarrow} \tau^{-1}\mathcal{N} \longrightarrow 0.
\end{equation}
be an AR sequence starting from $\mathcal{N}$.  One can show, using results from \cite[Section 1.1.2]{kussin}, that the only indecomposable bundles in ${\sf H}$ are of the form $\tau^{i}\mathcal{L}$ and $\tau^{i}\overline{\mathcal{L}}$, and  $\mathcal{E}$ is given according to the following table, where $(m,n)$ denotes the left-right dimension of $M$:
\begin{center}
\begin{tabular}{|c|c|c|}
\hline
$(m,n)$ & $\mathcal{N}$ & $\mathcal{E}$ \\ \hline \hline
$(2,2)$ & $\tau^{i}\mathcal{L}$ &  $\tau^{i}\overline{\mathcal{L}}^{\oplus 2}$ \\ \hline
$(2,2)$ & $\tau^{i}\overline{\mathcal{L}}$ &  $\tau^{i-1}{\mathcal{L}}^{\oplus 2}$ \\ \hline
$(1,4)$ & $\tau^{i}{\mathcal{L}}$ &  $\tau^{i}\overline{\mathcal{L}}$ \\ \hline
$(1,4)$ & $\tau^{i}\overline{\mathcal{L}}$ &  $\tau^{i-1}{\mathcal{L}}^{\oplus 4}$ \\ \hline
\end{tabular}
\end{center}
It is not hard to show, using the factorization property of AR sequences (see \cite[Proposition 2.1]{dlab}), that if we write $\mathcal{E} = \mathcal{P}^{\oplus n}$ where $\mathcal{P}$ is indecomposable, then the components of the map $h$ in (\ref{eqn.arsequence}) are a left basis for $\operatorname{Hom}(\mathcal{N},\mathcal{P})$.

We will also need the fact that, according to \cite[Section 1.1.2]{kussin},
\begin{equation} \label{eqn.vanish}
\operatorname{Hom}(\mathcal{L},\mathcal{N})=0
\end{equation}
if $\mathcal{N}$ is either $\tau^{i}\mathcal{L}$ or $\tau^{i}\overline{\mathcal{L}}$ and $i>0$.  Similarly,
\begin{equation} \label{eqn.vanish2}
\operatorname{Hom}(\overline{\mathcal{L}},\mathcal{N})=0
\end{equation}
if $\mathcal{N}$ is either $\tau^{i}\mathcal{L}$ and $i \geq 0$ or $\tau^{i}\overline{\mathcal{L}}$ and $i>0$.

Let $\mathcal{N}$ be an indecomposable bundle and let $\mathcal{P}$ be an indecomposable summand of $\mathcal{E}$ in (\ref{eqn.arsequence}).  Although it is mentioned in \cite{kussin} that
\begin{equation} \label{eqn.leftdual}
{}^{*}\operatorname{Hom}(\mathcal{N},\mathcal{P}) \cong \operatorname{Hom}(\mathcal{P}, \tau^{-1}\mathcal{N}),
\end{equation}
the $\operatorname{End}(\mathcal{N})$-module structure on $\operatorname{Hom}(\mathcal{P}, \tau^{-1}\mathcal{N})$ is not explicitly described.  Since it will be convenient for us to describe this structure, we reconstruct (\ref{eqn.leftdual}) in Proposition \ref{prop.phi}.  The proof of Proposition \ref{prop.phi} is an adaptation of \cite[Proposition 2.1]{dlab} to our setting.  We will need the following

\begin{lemma} \label{lemma.lstruct1}
Suppose $\mathcal{N}$ is an indecomposable object in ${\sf H}$.  If $a \in \operatorname{End}(\mathcal{N})$, then there exists a unique $f_{a} \in \operatorname{End}(\mathcal{E})$ such that the diagram
$$
\begin{CD}
\mathcal{N} & \overset{h}{\longrightarrow} & \mathcal{E}  \\
@V{a}VV  @VV{f_{a}}V \\
\mathcal{N} & \underset{h}{\longrightarrow} & \mathcal{E}
\end{CD}
$$
whose horizontals are from (\ref{eqn.arsequence}), commutes.  In addition, the function $a \mapsto f_{a}$ is a $k$-algebra homomorphism.
\end{lemma}

\begin{proof}
We write $\mathcal{E}=\mathcal{P}^{\oplus n}$, where $\mathcal{P}$ is an indecomposable bundle, and we use the fact that the components of $h=(f_{1},\ldots,f_{n})$ are a left basis for $\operatorname{Hom}(\mathcal{N},\mathcal{P})$.  Therefore, if $a \in \operatorname{End}(\mathcal{N})$, there exist $a_{ij}$ in $\operatorname{End}(\mathcal{P})$ such that
\begin{equation} \label{eqn.lr}
f_{i}a = \sum_{j} a_{ij}f_{j}.
\end{equation}
We let $f_{a} \in \operatorname{End}(\mathcal{P}^{\oplus n})$ denote the morphism uniquely determined by the fact that if $g_{i}:\mathcal{P} \longrightarrow \mathcal{P}^{\oplus n}$ denotes the $i$th inclusion, then $f_{a} g_{i} = (a_{1i},\ldots,a_{ni})$.  Uniqueness of $f_{a}$ follows from the fact that $\{f_{1},\ldots,f_{n}\}$ is a left basis for $\operatorname{Hom}(\mathcal{N},\mathcal{P})$.

The proof of the fact that the function $a \mapsto f_{a}$ is a $k$-algebra homomorphism is routine and omitted.
\end{proof}

\begin{prop} \label{prop.phi}
Let $\mathcal{P}$ be an indecomposable summand of $\mathcal{E}$ in (\ref{eqn.arsequence}).  There exists a $k$-algebra homomorphism
$$
\Phi:\operatorname{End}(\mathcal{N}) \longrightarrow \operatorname{End}(\tau^{-1}\mathcal{N})
$$
endowing $\operatorname{Hom}(\mathcal{P},\tau^{-1}\mathcal{N})$ with an $\operatorname{End}(\mathcal{N})-\operatorname{End}(\mathcal{P})$-bimodule structure.  With this structure, there is an isomorphism of $\operatorname{End}(\mathcal{N})-\operatorname{End}(\mathcal{P})$-bimodules
$$
\Psi: {}^{*}\operatorname{Hom}(\mathcal{N},\mathcal{P}) \longrightarrow \operatorname{Hom}(\mathcal{P},\tau^{-1}\mathcal{N}).
$$
Furthermore, $\Phi$ is induced by $\tau^{-1}$.
\end{prop}

\begin{proof}
We define the $k$-algebra homomorphism $\Phi:\operatorname{End}(\mathcal{N}) \longrightarrow \operatorname{End}(\tau^{-1}\mathcal{N})$ as follows:  given $a \in \operatorname{End}(\mathcal{N})$, we get a unique $f_{a} \in \operatorname{End}(\mathcal{E})$ such that the diagram
$$
\begin{CD}
\mathcal{N} & \overset{h}{\longrightarrow} & \mathcal{E} \\
@V{a}VV  @VV{f_{a}}V \\
\mathcal{N} & \underset{h}{\longrightarrow} & \mathcal{E}
\end{CD}
$$
commutes by Lemma \ref{lemma.lstruct1}.  It follows that there is a unique $g_{a} \in \operatorname{End}(\tau^{-1}\mathcal{N})$ making
\begin{equation} \label{eqn.ses}
\begin{CD}
\mathcal{E} & \overset{p}{\longrightarrow} & \tau^{-1}\mathcal{N}  \\
@VV{f_{a}}V @VV{g_{a}}V \\
\mathcal{E} & \overset{p}{\longrightarrow} & \tau^{-1}\mathcal{N}
\end{CD}
\end{equation}
commute.  We define $\Phi(a) := g_{a}$.  The proof that $\Phi(a)$ is a $k$-algebra homomorphism is routine and omitted.  Finally, the fact that $g_{a}=\tau^{-1}(a)$ is an easy application of \cite[Corollary 4.2]{newkussin} and we omit the details.

Next, we construct $\Psi:{}^{*}\operatorname{Hom}(\mathcal{N},\mathcal{P}) \longrightarrow \operatorname{Hom}(\mathcal{P},\tau^{-1}\mathcal{N})$.  To this end, suppose $\{\phi_{1},\ldots, \phi_{n}\}$ is the right basis for ${}^{*}\operatorname{Hom}(\mathcal{N},\mathcal{P})$ dual to a left basis $\{f_{1}, \ldots, f_{n}\}$ of $\operatorname{Hom}(\mathcal{N},\mathcal{P})$.  We define $\Psi$ by letting it send $\phi_{m}$ to $p \circ g_{m}$ where $g_{m}$ is inclusion of the $m$th factor of $\mathcal{P}$ in $\mathcal{E}$, and we extend right-linearly.

We need to show $\Psi$ is one-to-one, onto, and compatible with left multiplication.  We show first that $\Psi$ is one-to-one.  Suppose $a_{1}, \ldots, a_{n} \in \operatorname{End}(\mathcal{P})$ are such that
$$
p(\sum_{i} g_{i}a_{i})=0.
$$
Then $\sum_{i} g_{i}a_{i}:\mathcal{P} \longrightarrow \mathcal{E}$ factors through the kernel of $p$.  Since $\operatorname{Hom}(\mathcal{P},\mathcal{N})=0$ by (\ref{eqn.vanish}), it follows that $a_{1}=\cdots=a_{n}=0$.  Therefore, $\Psi$ is one-to-one.  The fact that $\Psi$ is onto will follow from the fact that the right dimension of $\operatorname{Hom}(\mathcal{P},\tau^{-1}\mathcal{N})$ is equal to $n$.  To prove this, we consider the long exact sequence resulting from applying the functor $\operatorname{Hom}(\mathcal{P},-)$ to the short exact sequence (\ref{eqn.arsequence}), and note that $0=D\operatorname{Hom}(\mathcal{N},\tau \mathcal{P}) \cong \operatorname{Ext}^{1}(\mathcal{P},\mathcal{N})$ by (\ref{eqn.vanish}) and (\ref{eqn.vanish2}).

We now show that $\Psi$ is compatible with the $\operatorname{End}(\mathcal{N})$-object structure.  To this end, we suppose the right action on the left basis $\{f_{1},\ldots, f_{n} \}$ is given by (\ref{eqn.lr}).  Then, by \cite[Lemma 3.4]{hart}, the left action on the right basis $\{\phi_{1}, \ldots, \phi_{n}\}$ of ${}^{*}\operatorname{Hom}(\mathcal{N},\mathcal{P})$ is given by
$$
a \phi_{i} = \sum_{j}\phi_{j} a_{ji}.
$$
Therefore, by definition of $\Psi$ we have
\begin{eqnarray*}
\Psi(a \cdot (\sum_{i}\phi_{i}b_{i})) & = & \Psi(\sum_{i,j}\phi_{j}a_{ji} b_{i}) \\
& = & p (\sum_{j}g_{j}(\sum_{i}a_{ji} b_{i})).
\end{eqnarray*}

On the other hand, if $b_{1},\ldots, b_{n} \in \operatorname{End}(\mathcal{P})$, then by definition of $f_{a}$ from Lemma \ref{lemma.lstruct1} we have
\begin{equation} \label{eqn.leftright}
f_{a}(b_{1},\ldots, b_{n})=(\sum_{i}a_{1i} b_{i},\ldots, \sum_{i}a_{ni} b_{i}).
\end{equation}
Therefore,
\begin{eqnarray*}
a \cdot (\sum_{i} p g_{i}b_{i}) & = & p (f_{a}(\sum_{i} g_{i}b_{i})) \\
& = & p (\sum_{j}g_{j}(\sum_{i}a_{ji} b_{i})).
\end{eqnarray*}
as desired.
\end{proof}
For the remainder of this section, $\Phi$ and $\Psi$ will refer to the maps defined in Proposition \ref{prop.phi}.

The following is an adaptation of \cite[Proposition 2.2]{dlab}.
\begin{lemma} \label{lemma.relations}
Let $\mathcal{P}$ be an indecomposable summand of $\mathcal{E}$ in (\ref{eqn.arsequence}).  Let $\{f_{1},\ldots,f_{n}\}$ denote a left basis for $\operatorname{Hom}(\mathcal{N},\mathcal{P})$, and let $\{\phi_{1}, \ldots, \phi_{n}\}$ denote the corresponding right dual basis for ${}^{*}\operatorname{Hom}(\mathcal{N},\mathcal{P})$.  Then, under the composition
\begin{eqnarray*}
{}^{*}\operatorname{Hom}(\mathcal{N},\mathcal{P}) \otimes_{\operatorname{End}(\mathcal{P})} \operatorname{Hom}(\mathcal{N},\mathcal{P}) & \overset{\Psi \otimes 1}{\longrightarrow} & \operatorname{Hom}(\mathcal{P},\tau^{-1}\mathcal{N}) \otimes_{\operatorname{End}(\mathcal{P})} \operatorname{Hom}(\mathcal{N},\mathcal{P}) \\
& \longrightarrow & \operatorname{Hom}(\mathcal{N},\tau^{-1}\mathcal{N})
\end{eqnarray*}
whose second arrow is induced by composition, the element $\sum_{i}\phi_{i} \otimes f_{i}$ goes to zero.
\end{lemma}

\begin{proof}
We let $g_{i}:\mathcal{P} \longrightarrow \mathcal{P}^{\oplus n}$ denote the $i$th inclusion.  By definition of $\Psi$, the element $\sum_{i}\phi_{i} \otimes f_{i}$ maps to $p (\sum_{i} g_{i} f_{i})$.  But this is zero as it is the composition of maps in (\ref{eqn.arsequence}).
\end{proof}

For the next result, recall the definition of $Q_{i}$ from Section \ref{section.ncsym}.
\begin{prop} \label{prop.01}
For $i \in \mathbb{Z}$ even, there is a canonical isomorphism of $\operatorname{End}(\overline{\mathcal{L}})-\operatorname{End}(\mathcal{L})$-bimodules
$$
\Psi_{i}:M^{i*} \longrightarrow \operatorname{Hom}(\tau^{\frac{i}{2}} \mathcal{L},\tau^{\frac{i}{2}}\overline{\mathcal{L}}),
$$
and for $i \in \mathbb{Z}$ odd there is a canonical isomorphism of $\operatorname{End}(\mathcal{L})-\operatorname{End}(\overline{\mathcal{L}})$-bimodules
$$
\Psi_{i}:M^{i*} \longrightarrow \operatorname{Hom}(\tau^{\frac{i+1}{2}}\overline{\mathcal{L}},\tau^{\frac{i-1}{2}}\mathcal{L}),
$$
where the bimodule structure on the codomain is induced by the appropriate power of $\tau$.  Furthermore, for $i \in \mathbb{Z}$, $Q_{i} \subset \operatorname{ker }(\Psi_{i} \otimes \Psi_{i+1})$.
\end{prop}

\begin{proof}
First we construct $\Psi_{i}$ for $i \geq 0$.  We proceed by induction on $i$.  If $i=0$, the isomorphism is equality.

Now, suppose $i+1$ is odd.  Then there is a composition, $\Psi_{i}'$ of isomorphisms
\begin{eqnarray*}
{}^{*}((M^{i*})^{*}) & \overset{\cong}{\longrightarrow} & M^{i*} \\
& \overset{\Psi_{i}}{\longrightarrow} & \operatorname{Hom}(\tau^{\frac{i}{2}}\mathcal{L},\tau^{\frac{i}{2}}\overline{\mathcal{L}}) \\
& \overset{\cong}{\longrightarrow} & \operatorname{Hom}(\tau^{\frac{i}{2}}\mathcal{L},\tau^{-1}\tau^{\frac{i+2}{2}}\overline{\mathcal{L}}) \\
& \overset{\cong}{\longrightarrow} & {}^{*}\operatorname{Hom}(\tau^{\frac{i+2}{2}}\overline{\mathcal{L}},\tau^{\frac{i}{2}}\mathcal{L})
\end{eqnarray*}
as follows: the first isomorphism is from Lemma \ref{lemma.eta}(1) with $\Gamma=\operatorname{id}_{M^{i+1*}}$, the third isomorphism is induced by a natural transformation
\begin{equation} \label{eqn.tau}
\tau^{\frac{i}{2}} \longrightarrow \tau^{-1}\tau^{\frac{i+2}{2}},
\end{equation}
and the fourth isomorphism is an application of Proposition \ref{prop.phi} with $\mathcal{N}=\tau^{\frac{i+2}{2}}\overline{\mathcal{L}}$ and with $\mathcal{P}=\tau^{\frac{i}{2}}\mathcal{L}$.

We claim that these isomorphisms are all compatible with the $\operatorname{End}(\overline{\mathcal{L}})-\operatorname{End}(\mathcal{L})$-bimodule structures.  For, by the naturality of (\ref{eqn.tau}), the third isomorphism is compatible with the $\operatorname{End}(\overline{\mathcal{L}})-\operatorname{End}(\mathcal{L})$-bimodule structures.  By Proposition \ref{prop.phi}, the fourth isomorphism is as $\operatorname{End}(\tau^{\frac{i+2}{2}}\overline{\mathcal{L}})-\operatorname{End}(\tau^{\frac{i}{2}}\mathcal{L})$-bimodules.  Hence, under the isomorphisms induced by the appropriate powers of $\tau$, this is an isomorphism of $\operatorname{End}(\overline{\mathcal{L}})-\operatorname{End}(\mathcal{L})$-bimodules.

It follows from Lemma \ref{lemma.eta}(2) that if $\{f_{1}, \ldots, f_{n}\}$ is a left basis for the bimodule  $\operatorname{Hom}(\tau^{\frac{i+2}{2}}\overline{\mathcal{L}},\tau^{\frac{i}{2}}\mathcal{L})$ and $\{\phi_{1},\ldots, \phi_{n}\}$ is the corresponding right dual basis for ${}^{*}\operatorname{Hom}(\tau^{\frac{i+2}{2}}\overline{\mathcal{L}},\tau^{\frac{i}{2}}\mathcal{L})$, then there exists a unique isomorphism of bimodules
$$
\Psi_{i+1}: M^{i+1*} \longrightarrow \operatorname{Hom}(\tau^{\frac{i+2}{2}}\overline{\mathcal{L}},\tau^{\frac{i}{2}}\mathcal{L})
$$
such that $(\Psi_{i}') \otimes \Psi_{i+1}$ maps $Q_{i,M}'$ to the submodule generated by $\sum_{i}\phi_{i} \otimes f_{i}$.

A similar construction yields $\Psi_{i+1}$ for $i+1>0$ even, and we omit the details.

Next, we construct $\Psi_{i}$ in case $i \leq 0$ inductively.  Suppose $i-1$ is odd, and suppose we have an isomorphism $\Psi_{i}:M^{i*} \longrightarrow \operatorname{Hom}(\tau^{\frac{i}{2}}\mathcal{L}, \tau^{\frac{i}{2}}\overline{\mathcal{L}})$.  Then we have a composition of isomorphisms
\begin{eqnarray*}
M^{i-1*} & \overset{=}{\longrightarrow} & {}^{*}(M^{i*}) \\
& \overset{\Psi_{i-1}'}{\longrightarrow} & {}^{*}\operatorname{Hom}(\tau^{\frac{i}{2}}\mathcal{L}, \tau^{\frac{i}{2}}\overline{\mathcal{L}}) \\
& \overset{\cong}{\longrightarrow} & \operatorname{Hom}(\tau^{\frac{i}{2}}\overline{\mathcal{L}},\tau^{-1}\tau^{\frac{i}{2}}\mathcal{L}) \\
& \overset{\cong}{\longrightarrow} & \operatorname{Hom}(\tau^{\frac{i}{2}}\mathcal{L},\tau^{\frac{i-2}{2}}\overline{\mathcal{L}})
\end{eqnarray*}
whose second arrow $\Psi_{i-1}'$ is the isomorphism $\Delta$ from Lemma \ref{lemma.eta}(2) with $\Gamma = \Psi_{i}$ and whose other arrows are defined as in the $i>0$ case.  As above, this isomorphism is compatible with the $\operatorname{End}({\mathcal{L}})-\operatorname{End}(\overline{\mathcal{L}})$-bimodule structures.

The construction of $\Psi_{i-1}$ in case $i \leq 0$, $i-1$ even is similar and we leave the details to the reader.

Now we prove the second part of the proposition.  We do this in case $i$ is even and $i \geq 0$.  The proofs in other cases are similar and omitted.  By the definition of $\Psi'_{i}$ and $\Psi_{i+1}$ above, and by Lemma \ref{lemma.relations}, we know that $Q_{i,M}'$ maps to zero under the composition
\begin{eqnarray*}
{}^{*}(M^{i+1*})\otimes M^{i+1*} & \overset{\Psi'_{i}\otimes \Psi_{i+1}}{\longrightarrow} & {}^{*}\operatorname{Hom}(\tau^{\frac{i+2}{2}}\overline{\mathcal{L}},\tau^{\frac{i}{2}}\mathcal{L}) \otimes \operatorname{Hom}(\tau^{\frac{i+2}{2}}\overline{\mathcal{L}},\tau^{\frac{i}{2}}\mathcal{L}) \\
& \overset{\cong}{\longrightarrow} &  \operatorname{Hom}(\tau^{\frac{i}{2}}\mathcal{L},\tau^{-1}\tau^{\frac{i+2}{2}}\overline{\mathcal{L}}) \otimes \operatorname{Hom}(\tau^{\frac{i+2}{2}}\overline{\mathcal{L}},\tau^{\frac{i}{2}}\mathcal{L}) \\
& \longrightarrow & \operatorname{Hom}(\tau^{\frac{i+2}{2}}\overline{\mathcal{L}},\tau^{-1}\tau^{\frac{i+2}{2}}\overline{\mathcal{L}})
\end{eqnarray*}
whose second arrow is from Proposition \ref{prop.phi} and whose last arrow is composition.  Therefore, by definition of $\Psi_{i}$, $Q_{i,M}'$ maps to zero under the composition
\begin{eqnarray*}
{}^{*}(M^{i+1*})\otimes M^{i+1*} & \overset{\cong}{\longrightarrow} & M^{i*}\otimes M^{i+1*} \\
& \overset{\Psi_{i}\otimes \Psi_{i+1}}{\longrightarrow} & \operatorname{Hom}(\tau^{\frac{i}{2}}\mathcal{L},\tau^{\frac{i}{2}}\overline{\mathcal{L}}) \otimes \operatorname{Hom}(\tau^{\frac{i+2}{2}}\overline{\mathcal{L}},\tau^{\frac{i}{2}}\mathcal{L}) \\
& \longrightarrow & \operatorname{Hom}(\tau^{\frac{i+2}{2}}\overline{\mathcal{L}},\tau^{\frac{i}{2}}\overline{\mathcal{L}})
\end{eqnarray*}
whose first arrow is from Lemma \ref{lemma.eta}(1) with $\Gamma=\operatorname{id}_{M^{i+1*}}$ and whose last arrow is composition.  But this implies, by Lemma \ref{lemma.eta}(1), that $Q_{i,M}$ goes to zero under the composition of the last three maps above, completing the proof in this case.
\end{proof}

\begin{lemma} \label{lemma.dimform}
Let $\mathcal{F}$ and $\mathcal{G}$ be indecomposable bundles in ${\sf H}$, and let $\operatorname{rdim}(\operatorname{Hom}(\mathcal{F},\mathcal{G}))$ denote the dimension of $\operatorname{Hom}(\mathcal{F},\mathcal{G})$ as a $\operatorname{End}(\mathcal{F})$-vector space.  If $i \geq 0$ is an integer then we have
\begin{center}
\begin{tabular}{|c|c|c|c|}
\hline
$\mathcal{F}$ & $\mathcal{G}$ & $\operatorname{rdim}(\operatorname{Hom}(\mathcal{F},\mathcal{G}))$ in $(1,4)$ case & $\operatorname{rdim}(\operatorname{Hom}(\mathcal{F},\mathcal{G}))$ in $(2,2)$ case \\ \hline \hline
$\mathcal{L}$ & $\tau^{-i}\mathcal{L}$ &  $2i+1$ & $2i+1$ \\ \hline
$\overline{\mathcal{L}}$ & $\tau^{-i}\mathcal{L}$ & $i$ & $2i$ \\ \hline
$\overline{\mathcal{L}}$ & $\tau^{-i}\overline{\mathcal{L}}$ & $2i+1$ & $2i+1$ \\ \hline
$\mathcal{L}$ & $\tau^{-i}\overline{\mathcal{L}}$ & $4i+4$ & $2i+2$ \\ \hline
\end{tabular}
\end{center}
If $i<0$ all of the above $\operatorname{Hom}$ spaces are zero.
\end{lemma}

\begin{proof}
We prove the formulas in case $M$ is a $(1,4)$-bimodule.  The proof in the case that $M$ is a $(2,2)$-bimodule is similar and omitted.  The first formula is \cite[Section 1.1.15]{kussin}.  We prove the second and third formulas by induction on $i$.  The $i=0$ case follows from (\ref{eqn.vanish2}).  Now suppose $i \geq 0$ and consider the AR sequences
\begin{equation} \label{eqn.arnew1}
0 \longrightarrow \tau^{-i}\mathcal{L} \longrightarrow \tau^{-i}\overline{\mathcal{L}} \longrightarrow \tau^{-(i+1)}\mathcal{L} \longrightarrow 0
\end{equation}
and
\begin{equation} \label{eqn.arnew2}
0 \longrightarrow \tau^{-i}\overline{\mathcal{L}} \longrightarrow \tau^{-(i+1)}\mathcal{L}^{\oplus 4} \longrightarrow \tau^{-(i+1)}\overline{\mathcal{L}} \longrightarrow 0.
\end{equation}
Applying $\operatorname{Hom}(\overline{\mathcal{L}},-)$ to (\ref{eqn.arnew1}) yields a short exact sequence
$$
0 \longrightarrow \operatorname{Hom}(\overline{\mathcal{L}},\tau^{-i}\mathcal{L}) \longrightarrow \operatorname{Hom}(\overline{\mathcal{L}},\tau^{-i}\overline{\mathcal{L}}) \longrightarrow \operatorname{Hom}(\overline{\mathcal{L}},\tau^{-(i+1)}\mathcal{L}) \longrightarrow 0
$$
by (\ref{eqn.vanish}), and the second formula in the $i+1$ case follows from the induction hypothesis applied to this sequence.

The third formula in the $i+1$ case now follows by applying $\operatorname{Hom}(\overline{\mathcal{L}},-)$ to (\ref{eqn.arnew2}) and using (\ref{eqn.vanish}), then employing the $i+1$ case of the second formula and the induction hypothesis.  Now the second and third formulas follow by induction.

Finally, the fourth formula follows in a similar fashion by applying $\operatorname{Hom}(\mathcal{L},-)$ to (\ref{eqn.arnew1})  and using the first formula.

The last statement follows from (\ref{eqn.vanish}) and (\ref{eqn.vanish2}).
\end{proof}
We now define a sequence of objects in ${\sf H}$ which will be used to construct a $\mathbb{Z}$-algebra coordinate ring for ${\sf H}$.  For $n \in \mathbb{Z}$, we define
\begin{equation} \label{eqn.oi}
\mathcal{O}(n) := \begin{cases} \tau^{\frac{-n}{2}}\overline{\mathcal{L}} & \mbox{if $n$ is even} \\ \tau^{\frac{-(n+1)}{2}} \mathcal{L} & \mbox{if $n$ is odd.} \end{cases}
\end{equation}
We define a $\mathbb{Z}$-algebra $H$ by setting
$$
H_{ij} = \begin{cases} \operatorname{Hom}(\mathcal{O}(-j),\mathcal{O}(-i)) & \mbox{if $j \geq i$} \\ 0 & \mbox{if $i>j$}\end{cases}
$$
and defining multiplication as composition.

The next result follows immediately from Lemma \ref{lemma.dimform}.
\begin{cor} \label{cor.hcount}
Suppose $j \geq i$.  If $M$ is a $(1,4)$-bimodule, then the right dimension of $H_{ij}$ is
$$
\begin{cases} j-i+1 & \mbox{if $i$ and $j$ have the same parity,} \\ \frac{j-i+1}{2} & \mbox{if $i$ is odd and $j$ is even, and} \\ 2j-2i+2 & \mbox{if $i$ is even and $j$ is odd.} \end{cases}
$$

If $M$ is a $(2,2)$-bimodule, then the right dimension of $H_{ij}$ is $j-i+1$.
\end{cor}

In order to prove Proposition \ref{prop.newdimcount}, we will need the following technical lemma.
\begin{lemma} \label{lemma.inq}
Let $M$ be a $(1,4)$-bimodule, and write $\mathbb{S}^{n.c.}(M)_{ij}=T_{ij}/R_{ij}$, where $T$ and $R$ are defined in Section \ref{section.ncsym}.  Suppose $i,j\in \mathbb{Z}$ are such that $j$ is odd and $j > i$.  Finally, suppose $v \in T_{ij}$ has the property that there exist two right-independent vectors $g_{1},g_{2} \in  T_{j j+1}$ such that $v \otimes g_{l} \in R_{ij+1}$ for $l=1,2$.  Then $v \in R_{ij}$.
\end{lemma}

\begin{proof}
We proceed by induction on $j-i$.  Suppose $j-i=1$, so that $v \in M^{j*}$.  Let $\phi$ be a nonzero element of $M^{j*}$ so that $\phi$ is a right basis for $M^{j*}$, and let $\phi^{*} \in M^{j+1*}$ denote the corresponding right dual.  Then, for $l=1,2$, there exist $a, b_{l} \in \operatorname{End}(\overline{\mathcal{L}})$ and $c_{l} \in \operatorname{End}(\mathcal{L})$ such that $v=\phi a$, 
$$
g_{l}=b_{l} \phi^{*},
$$ 
and
$$
\phi a \otimes b_{l} \phi^{*}=\phi \otimes \phi^{*} c_{l}.
$$
Therefore, for $l=1,2$, $ab_{l}\phi^{*}=\phi^{*}c_{l}$.  Thus, if $a \neq 0$, then $c_{l} \neq 0$, and so for $l=1,2$, $g_{l}c_{l}^{-1}=a^{-1}\phi^{*}$.  This contradicts the right independence of $g_{1},g_{2}$.  It follows that $a=0$ and hence $v=0$, so that $v \in R_{ii+1}$.

Now suppose $j-i>1$ and retain the notation in the previous paragraph.  Since, for $l=1,2$,
$$
v \otimes g_{l} \in R_{i j+1}=T_{ij-1} \otimes Q_{j}+R_{ij} \otimes M^{j+1*},
$$
there exist $h_{l} \in T_{ij-1}$ such that
$$
v \otimes g_{l} - h_{l} \otimes \phi \otimes \phi^{*} \in R_{ij} \otimes M^{j+1*}.
$$
Since $\phi^{*}$ is a left basis for $M^{j+1*}$, it follows that, for $l=1,2$,
\begin{equation} \label{eqn.cross}
vb_{l} - h_{l} \otimes \phi \in R_{ij},
\end{equation}
and hence
\begin{equation} \label{eqn.cross2}
h_{1} \otimes \phi b_{1}^{-1} - h_{2} \otimes \phi b_{2}^{-1} \in R_{ij}.
\end{equation}
It is straightforward to check that the set $\{f_{1}:=\phi b_{1}^{-1}, f_{2}:=\phi b_{2}^{-1}\}$ is left independent so that it can be extended to a left basis $\{f_{n}\}_{n=1}^{4}$ of $M^{j*}$.  If $j=i+2$, then (\ref{eqn.cross2}) implies that $h_{1}=0$ so that (\ref{eqn.cross}) implies that $v \in R_{ij}$.  If $j>i+2$, then (\ref{eqn.cross2}) implies there is an $h \in T_{i j-2}$ such that
$$
h_{1} \otimes f_{1} - h_{2} \otimes f_{2}-h \otimes \sum_{l}{}^{*}f_{l} \otimes f_{l} \in R_{i j-1} \otimes M^{j*}
$$
where $\{{}^{*}f_{n}\}$ is the dual basis to $\{f_{n}\}$.  Therefore, $h \otimes {}^{*}f_{3}, h \otimes {}^{*}f_{4}$ and $h_{1}-h \otimes {}^{*}f_{1}$ are in $R_{i j-1}$.  By induction, $h \in R_{i j-2}$, so that $h_{1} \in R_{i j-1}$.  Therefore, by (\ref{eqn.cross}), $v \in R_{ij}$.
\end{proof}

\begin{prop} \label{prop.newdimcount}
The right dimension of $H_{ij}$ equals that of $\mathbb{S}^{n.c.}(M)_{ij}$ for all integers $i,j$.
\end{prop}

\begin{proof}
If $M$ is a $(2,2)$-bimodule, the result follows from Corollary \ref{cor.hcount} and \cite[Theorem 6.1.2(1)]{vandenbergh}.

Now, suppose that $M$ is a $(1,4)$-bimodule.  In this case, we prove that, as in \cite[Theorem 6.1.2 (2)]{vandenbergh}, the exact sequence induced by multiplication
\begin{equation} \label{eqn.euler}
\mathbb{S}^{n.c.}(M)_{ij-1} \otimes Q_{j} \rightarrow \mathbb{S}^{n.c.}(M)_{ij} \otimes M^{j+1*} \rightarrow \mathbb{S}^{n.c.}(M)_{i,j+1} \rightarrow 0
\end{equation}
is exact on the left for all $i \in \mathbb{Z}$.  The result will follow from this and induction in light of Corollary \ref{cor.hcount}, as the reader can check.

We retain the notation from the statement of Lemma \ref{lemma.inq}.  Left exactness of (\ref{eqn.euler}) is equivalent to the equality
\begin{equation} \label{eqn.newintersect}
R_{ij} \otimes M^{j+1*} \cap T_{i j-1} \otimes Q_{j} = R_{i j-1} \otimes Q_{j}
\end{equation}
for all $i,j \in \mathbb{Z}$.  For $j<i+2$, both sides of (\ref{eqn.newintersect}) are zero.  Now suppose $j \geq i+2$.  There are two cases to consider.

\medskip
\medskip
{\it Case 1: $j$ is odd}.  Let $\phi \in M^{j*}$ be nonzero, so that it is a right basis for $M^{j*}$, and let $\phi^{*} \in M^{j+1*}$ denote it's right dual.  Extend $f_{1} := \phi$ to a left basis $\{f_{1},\ldots, f_{4}\}$ of $M^{j*}$, and let $\{{}^{*}f_{1},\ldots,{}^{*}f_{4}\}$ denote the corresponding dual right basis for $M^{j-1*}$.

Suppose $v \in T_{i j-1}$ is such that $v \otimes \phi \otimes \phi^{*}$ is in the left-hand side of (\ref{eqn.newintersect}). Now we consider two sub-cases:  if $j=i+2$, then there exists $a \in \operatorname{End }(\overline{\mathcal{L}})$ such that
$$
v \otimes \phi \otimes \phi^{*} = a (\sum_{l}{}^{*}f_{l} \otimes f_{l}) \otimes \phi^{*}.
$$
Since $\phi^{*}$ is a left basis for $M^{j+1*}$, we conclude that
$$
v \otimes \phi = v \otimes f_{1} = a (\sum_{l}{}^{*}f_{l} \otimes f_{l}).
$$
Therefore, since $\{f_{l}\}_{l=1}^{4}$ is a left basis for $M^{j*}$ we conclude that $a$, and hence $v$, equals zero.

If $j>i+2$, then there exists $w \in R_{ij-1} \otimes M^{j*}$ and $u \in T_{ij-2}$ such that
$$
v \otimes \phi \otimes \phi^{*} = w \otimes \phi^{*}+ u \otimes (\sum_{l}{}^{*}f_{l} \otimes f_{l}) \otimes \phi^{*}.
$$
Therefore,
$$
v \otimes f_{1} - u \otimes (\sum_{l}{}^{*}f_{l} \otimes f_{l}) \in R_{ij-1} \otimes M^{j*},
$$
so that, since $\{f_{l}\}_{l=1}^{4}$ is a left basis for $M^{j*}$, we conclude that
\begin{equation} \label{eqn.newdelta1}
v-u \otimes {}^{*}f_{1} \in R_{ij-1}
\end{equation}
and
\begin{equation} \label{eqn.newdelta2}
u \otimes {}^{*}f_{j} \in R_{i j-1}
\end{equation}
for $j \neq 1$.  It now follows from (\ref{eqn.newdelta2}) and Lemma \ref{lemma.inq} that $u \in R_{i j-2}$ so that (\ref{eqn.newdelta1}) implies that $v \in R_{i j-1}$ as desired, completing the proof in the case that $j$ is odd.

\medskip
\medskip
{\it Case 2: $j$ is even}.  Let $\phi \in M^{j*}$ be nonzero so that it forms a left basis for $M^{j*}$, let ${}^{*}\phi$ be its left dual in $M^{j-1*}$, let $\{f_{1}:= \phi, f_{2},f_{3}, f_{4}\}$ denote a right basis for $M^{j*}$, and let $\{f_{1}^{*}, \ldots, f_{4}^{*}\}$ denote the corresponding dual left basis for $M^{j+1*}$.  Finally, let $a_{l} \in \operatorname{End}(\overline{\mathcal{L}})$ be such that 
$$
a_{l}\phi=f_{l}.
$$

Suppose that $v \in T_{i j-1}$ is such that $v \otimes \sum_{l}f_{l} \otimes f_{l}^{*}$ is an element of the left-hand side of (\ref{eqn.newintersect}).  As in the case when $j$ is odd, there are two sub-cases to consider.  We leave the $j=i+2$ case as a straightforward exercise.  Next, suppose $j>i+2$.   Arguing as in the $j$ odd case, we conclude that, for $1 \leq l \leq 4$, there exist $w_{l} \in T_{i j-2}$ such that
$$
v \otimes a_{l} \phi- w_{l} \otimes {}^{*}\phi \otimes \phi \in R_{ij-1} \otimes M^{j*}.
$$
Since $\phi$ is a left basis for $M^{j*}$, this implies that, for all $l$,
\begin{equation} \label{eqn.triplecross}
va_{l}-w_{l}\otimes {}^{*}\phi \in R_{i j-1}.
\end{equation}
Therefore, for $1 \leq m \leq n \leq 4$,
\begin{equation} \label{eqn.lasty}
w_{m} \otimes {}^{*}\phi a_{m}^{-1} - w_{n} \otimes {}^{*}\phi a_{n}^{-1} \in R_{ij-1}.
\end{equation}
As one can check, $\{g_{l}:={}^{*}\phi a_{l}^{-1}\}_{l=1}^{4}$ is left independent.  It follows easily that if $j=i+3$, then $w_{m}=0$, so that (\ref{eqn.triplecross}) implies $v \in R_{ij-1}$.   Thus, it remains to prove the result in case $j>i+3$.  To this end, if we let $\{{}^{*}g_{1},\ldots, {}^{*}g_{4}\}$ denote the corresponding dual basis of $\{g_{1}, \ldots, g_{4}\}$, then (\ref{eqn.lasty}) implies that for $1 \leq m \leq n \leq 4$, there exists $w \in M^{j-3*}$ such that
$$
w_{m} \otimes g_{m}-w_{n} \otimes g_{n} - w \otimes (\sum_{l}{}^{*}g_{l} \otimes g_{l}) \in R_{i j-2} \otimes V^{j-1*}.
$$
Thus, for $m \neq n$ and $l \neq m,n$, $w \otimes {}^{*}g_{l} \in R_{ij-2}$, so that Lemma \ref{lemma.inq} implies that $w \in R_{ij-3}$, and $w_{m} \otimes g_{m} - w \otimes {}^{*}g_{m} \otimes g_{m} \in R_{i j-2} \otimes V^{j-1*}$ so that $w_{m} \otimes g_{m} \in R_{i j-2} \otimes V^{j-1*}$.  Thus $w_{m} \in R_{i j-2}$ and so (\ref{eqn.triplecross}) implies that $v \in R_{ij-1}$ as desired.
\end{proof}

In order to state the next proposition, we need to introduce some terminology.  Following \cite[Section 2]{az} and \cite[Section 2]{quadrics}, if $A$ is a $\mathbb{Z}$-algebra, we let ${\sf Gr }A$ denote the category of graded right $A$-modules and we let ${\sf Tors }A$ denote the full subcategory of ${\sf Gr }A$ consisting of objects which are direct limits of right bounded modules.  We let ${\sf Proj }A$ denote the quotient of ${\sf Gr }A$ by ${\sf Tors }A$.  We also need to recall from  \cite[Definition 2.1.4]{staff} that a sequence of objects $\{\mathcal{P}(n) \}_{n \in \mathbb{Z}}$ in a locally noetherian category ${\sf C}$ is {\it ample} if
\begin{enumerate}
\item{} for every object $\mathcal{A} \in {\sf C}$, there are positive integers $l_{1},\ldots, l_{p}$ and an epimorphism $\oplus_{i=1}^{p}\mathcal{P}(-l_{i}) \longrightarrow \mathcal{A}$, and

\item{} if $\mathcal{A}$, $\mathcal{B}$ are objects in ${\sf C}$ and $f:\mathcal{A} \longrightarrow \mathcal{B}$ is an epimorphism, then the induced map $\operatorname{Hom}_{\sf C}(\mathcal{P}(-n),\mathcal{A}) \longrightarrow \operatorname{Hom}_{\sf C}(\mathcal{P}(-n),\mathcal{B})$ is surjective for all $n>>0$.
\end{enumerate}

\begin{prop} \label{prop.firstequiv}
The sequence $\{\mathcal{O}(n)\}_{n \in \mathbb{Z}}$ defined in (\ref{eqn.oi}) is an ample sequence of noetherian objects in ${\sf H}$.  Therefore, $H$ is noetherian (i.e. the category ${\sf Gr }H$ is locally noetherian) and
$$
\tilde{{\sf H}} \equiv {\sf Proj }H.
$$
\end{prop}

\begin{proof}
The first part of the proposition follows from the second part by \cite[Theorem 11.1.1]{staff}. We now prove the second part. We first recall that $(\mathcal{L},\tau^{-1})$ is an ample pair \cite[p. 3]{kussin}, i.e., the sequence $\{\tau^{-n}\mathcal{L}\}_{n \in \mathbb{Z}}$ is an ample sequence, so that part (1) of the definition of ampleness for $\{\mathcal{O}(n)\}_{n \in \mathbb{Z}}$ follows immediately.  It also follows from this that if
$$
0 \longrightarrow \mathcal{C} \longrightarrow \mathcal{A} \longrightarrow \mathcal{B} \longrightarrow 0
$$
is a short exact sequence in ${\sf H}$, then $\operatorname{Ext}^{1}(\tau^{n}\mathcal{L},C)=0=\operatorname{Ext}^{1}(\tau^{n+1}\mathcal{L},C)$ for some positive $n$.  Therefore, by applying $\operatorname{Hom}(-,C)$ to the AR sequence
$$
0 \longrightarrow \tau^{n+1}\mathcal{L} \longrightarrow \tau^{n+1}\overline{\mathcal{L}}^{\oplus m} \longrightarrow \tau^{n}\mathcal{L} \longrightarrow 0
$$
where $m=1$ if $M$ is a $(1,4)$-bimodule and $m=2$ if $M$ is a $(2,2)$-bimodule, we conclude that
$$
\operatorname{Ext}^{1}(\tau^{n+1}\overline{\mathcal{L}},C)=0.
$$
Part (2) in the definition of ampleness for $\{\mathcal{O}(n)\}_{n \in \mathbb{Z}}$ follows.
\end{proof}

\begin{thm} \label{thm.almostring}
Suppose ${\sf H}$ is a homogeneous noncommutative curve of genus zero with underlying bimodule $M$ such that $\operatorname{End}(\mathcal{L})$ and $\operatorname{End}(\overline{\mathcal{L}})$ are commutative.  Then there exists a $k$-linear isomorphism of $\mathbb{Z}$-algebras
$$
\mathbb{S}^{n.c.}(M) \longrightarrow H.
$$
Therefore, $\mathbb{P}^{n.c.}(M) \equiv \tilde{{\sf H}}$.

Conversely, if $K$ and $L$ are finite extensions of $k$ and $N$ is a $k$-central $K-L$-bimodule of left-right dimension $(1,4)$ or $(2,2)$, then there exists a homogenous noncommutative curve ${\sf H}$ with underlying bimodule $N$ such that $\mathbb{P}^{n.c.}(N) \equiv \tilde{\sf H}$.
\end{thm}

\begin{proof}
We define a $\mathbb{Z}$-algebra isomorphism
$$
\Lambda: \mathbb{S}^{n.c.}(M) \longrightarrow H
$$
as follows.  We let
$$
\Lambda_{ij}=0
$$ for $i>j$.  For $i$ odd we let
$$
\Lambda_{ii}:\operatorname{End}(\mathcal{L}) \longrightarrow \operatorname{End}(\mathcal{O}(-i))=\operatorname{End}(\tau^{\frac{i-1}{2}}\mathcal{L})
$$
be induced by $\tau^{\frac{i-1}{2}}$, and for $i$ even we let
$$
\Lambda_{ii}:\operatorname{End}(\overline{\mathcal{L}}) \longrightarrow \operatorname{End}(\tau^{\frac{i}{2}}\overline{\mathcal{L}})
$$
be induced by $\tau^{\frac{i}{2}}$.  For $i,j \in \mathbb{Z}$ with $i<j$, we let
$$
\Lambda_{ij}:\mathbb{S}^{n.c.}(M)_{ij} \longrightarrow \operatorname{Hom}(\mathcal{O}(-j),\mathcal{O}(-i))
$$
denote the map induced by the composition
\begin{eqnarray*}
M^{i*}\otimes \cdots \otimes M^{j-1*} & \rightarrow & \operatorname{Hom}(\mathcal{O}(-i-1),\mathcal{O}(-i)) \otimes \cdots \otimes \operatorname{Hom}(\mathcal{O}(-j),\mathcal{O}(-j+1)) \\
& \rightarrow & \operatorname{Hom}(\mathcal{O}(-j),\mathcal{O}(-i))
\end{eqnarray*}
whose first arrow is from Proposition \ref{prop.01} and whose second arrow is composition.  By Proposition \ref{prop.01}, this map does indeed factor through $\mathbb{S}^{n.c.}(M)_{ij}$.  To show that $\Lambda_{ij}$ is an isomorphism in this case, we first note that the factorization property of AR sequences implies that the second map in the above composition is surjective.  Therefore, $\Lambda_{ij}$ is surjective.  The fact that $\Lambda_{ij}$ is injective now follows from Proposition \ref{prop.newdimcount}.

To complete the proof of the first part of the theorem, it remains to check that $\Lambda$ is compatible with multiplication.  This straightforward exercise is left to the reader.

The second part of the theorem follows from the fact that given $N$ as in the hypotheses, there exists an ${\sf H}$ with underlying bimodule $N$ by \cite[Section 0.5.2]{kussin}.  Then one applies the first part of the theorem.
\end{proof}
The next result follows immediately from Theorem \ref{thm.main}.
\begin{corollary} \label{newcor.main}
A homogeneous noncommutative curve of genus zero, ${\sf H}$, with underlying bimodule $M$ is arithmetic if and only if $\operatorname{End}(\mathcal{L})$ is $k$-isomorphic to $\operatorname{End}(\overline{\mathcal{L}})$ and each is commutative.  In this case, $\mathbb{P}^{n.c.}(M) \equiv \tilde{{\sf H}}$.
\end{corollary}

\section{Applications}
In this section, we assume $V$ and $W$ are $(2,2)$-bimodules over two copies of the same field $K$, which is of finite dimension over $k$.  We will need to recall three types of canonical equivalences between noncommutative projective lines, studied in \cite{nyman}.  If $\delta, \epsilon \in \operatorname{Gal}(K/k)$ and we define
$$
\zeta_{i} = \begin{cases} \delta \mbox{ if $i$ is even} \\ \epsilon \mbox{ if $i$ is odd,} \end{cases}
$$
then twisting by the sequence $\{K_{\zeta_{i}}\}_{i \in \mathbb{Z}}$ (see Section \ref{section.questions} for the definition of $K_{\zeta_{i}}$) in the sense of \cite[Section 3.2]{vandenbergh} induces an equivalence of categories
$$
T_{\delta,\epsilon}:\mathbb{P}^{n.c.}(V) \longrightarrow \mathbb{P}^{n.c.}(K_{\delta^{-1}} \otimes_{K} V \otimes_{K} K_{\epsilon}).
$$
Next, let $\phi:V \longrightarrow W$ denote an isomorphism of $K-K$-bimodules.  Then $\phi$ induces an isomorphism $\mathbb{S}^{n.c.}(V) \longrightarrow \mathbb{S}^{n.c.}(W)$, which in turn induces an equivalence
$$
\Phi:\mathbb{P}^{n.c.}(V) \longrightarrow \mathbb{P}^{n.c.}(W).
$$
Finally, shifting graded modules by an integer, $i$, induces an equivalence
$$
[i]:\mathbb{P}^{n.c.}(V) \longrightarrow \mathbb{P}^{n.c.}(V^{i*}),
$$
and we have the following classification of equivalences between noncommutative projective lines (\cite[Theorem 6.9]{nyman}):

\begin{thm} \label{thm.auto}
Suppose $k$ is a perfect field with $\operatorname{char }k \neq 2$ and let $F:\mathbb{P}^{n.c.}(V) \longrightarrow \mathbb{P}^{n.c.}(W)$ be a $k$-linear equivalence.  Then there exists $\delta, \epsilon \in \mbox{Gal }(K/k)$, an isomorphism $\phi:K_{\delta^{-1}} \otimes_{K} V \otimes_{K} K_{\epsilon} \longrightarrow W^{i^*}$ inducing an equivalence $\Phi: \mathbb{P}^{n.c.}(K_{\delta^{-1}} \otimes_{K} V \otimes_{K} K_{\epsilon}) \longrightarrow \mathbb{P}^{n.c.}(W^{i*})$ and an integer $i$ such that
$$
F \cong [-i] \circ \Phi  \circ T_{\delta, \epsilon}.
$$
Furthermore, $\delta$, $\epsilon$ and $i$ are unique up to natural equivalence, while $\Phi$ is naturally equivalent to $\Phi'$ if and only if there exist nonzero $a, b \in K$ such that $\phi' \phi^{-1}(w) = a \cdot w \cdot b$ for all $w \in W^{i*}$.
\end{thm}

\subsection{Applications to arithmetic noncommutative projective lines}

\subsubsection{$\mathbb{S}^{n.c.}(V)$ is a domain}
In \cite[Theorem 3.7]{nyman}, homological techniques are employed to prove that if $V$ is a $K-K$-bimodule of left-right dimension $(2,2)$, then $\mathbb{S}^{n.c.}(V)$ is a domain in the sense that if $x \in \mathbb{S}^{n.c.}(V)_{ij}$ and $y \in \mathbb{S}^{n.c.}(V)_{jl}$ then $xy=0$ implies that $x=0$ or $y=0$.  Theorem \ref{thm.main} allows us to produce a much shorter proof of a generalization.

\begin{prop} \label{prop.domain}
Suppose $K$ and $L$ are finite extensions of $k$ and $N$ is a $k$-central $K-L$-bimodule of left-right dimension $(2,2)$.  Then $\mathbb{S}^{n.c.}(N)$ is a domain in the sense above.
\end{prop}

\begin{proof}
Theorem \ref{thm.main} implies it is sufficient to prove the result for the isomorphic ring $H$.  But in this case the result follows immediately from \cite[Lemma 1.3]{lenzingreiten}.
\end{proof}

\subsubsection{The Bondal-Orlov Theorem}
It is known that homogeneous noncommutative curves of genus zero have a Bondal-Orlov type reconstruction theorem.  Therefore, it follows from Corollary \ref{cor.main} that the same holds for arithmetic noncommutative projective lines.  Since we could not find a reference, we include a proof of this fact below.  To this end, we recall from \cite[Section 3]{lenzing} that the repetitive category of an abelian category $\sf{A}$ is the additive closure of the union of disjoint copies of $\sf{A}$ (labeled ${\sf A}[n]$ for $n \in \mathbb{Z}$ with objects labeled $A[n]$) with morphisms given by $\operatorname{Hom}(A[m],B[n]) = \operatorname{Ext}^{n-m}_{\sf A}(A,B)$ and composition given by Yoneda product of extensions.
The notation for the second part of the next theorem is defined before, and in the statement of, Theorem \ref{thm.auto}.
\begin{thm} \label{thm.bondalorlov}
Let ${\sf H}_{1}$ and ${\sf H}_{2}$ be homogeneous noncommutative curves of genus zero.  If there is a $k$-linear triangulated equivalence $F:D^{b}({\sf H}_{1}) \longrightarrow D^{b}({\sf H}_{2})$, then $F \cong T^{i} \circ G$ where $G$ is induced by a $k$-linear equivalence $\underline{G}:{\sf H}_{1} \longrightarrow {\sf H}_{2}$ and $T$ is the translation functor in the derived category.  In particular, if $\tilde{\sf H}_{1}=\mathbb{P}^{n.c.}(V)$, $\tilde{\sf H}_{2}=\mathbb{P}^{n.c.}(W)$, $k$ is perfect and $\operatorname{char }k \neq 2$, then $\underline{G} \cong  [-i] \circ \Phi  \circ T_{\delta, \epsilon}$.
\end{thm}

\begin{proof}
By \cite[Theorem 3.1]{lenzing}, we may identify $D^{b}({\sf H}_{1})$ and $D^{b}({\sf H}_{2})$ with their repetitive category.  Furthermore, since ${\sf H}_{1}$ and ${\sf H}_{2}$ are noetherian hereditary categories having Serre duality, we know that there exists an integer $i$ such that all objects from ${\sf H}_{1}[0]$ may be identified, via $F$, as objects in the additive closure of ${\sf H}_{2}[i]$ and ${\sf H}_{2}[i+1]$ in the repetitive category of ${\sf H}_{2}$ \cite[Proposition 7.2]{lenzing}.  Therefore, composing $F$ with an iterate of the translation functor, we may assume objects from ${\sf H}_{1}[0]$ may be identified with objects in the additive closure of ${\sf H}_{2}[0]$ and ${\sf H}_{2}[1]$.  By definition of morphism in the repetitive category,
\begin{equation} \label{eqn.homrep}
\operatorname{Hom}({\sf H}_{2}[1],{\sf H}_{2}[0]) = 0.
\end{equation}

We claim that either the image of all objects of ${\sf H}_{1}[0]$ are in ${\sf H}_{2}[0]$ or the image of all objects of ${\sf H}_{1}[0]$ are in ${\sf H}_{2}[1]$.  Given the claim, then by composing with a translation we may assume the former is true, so that there is an integer $j$ such that $T^{j} \circ F$ restricts to a $k$-linear equivalence from ${\sf H}_{1}$ to ${\sf H}_{2}$.  The result will follow.  Therefore, to complete the proof, we prove the claim.

Since $F$ preserves indecomposability, every indecomposable of ${\sf H}_{1}[0]$ is sent to an indecomposable in ${\sf H}_{2}[0]$ or an indecomposable in ${\sf H}_{2}[1]$.  We will use this fact without comment throughout the proof.  We will also use the notation introduced by (\ref{eqn.oi}).
\newline
{\it Step 1:  We show that if $\mathcal{S} \in {\sf H}_{1}[0]$ is simple, then $F(S)$ has finite length.}  For, if not, then $F(S) \cong \mathcal{O}(j)$ for some $j \in \mathbb{Z}$ by \cite[Proposition 1.1]{lenzingreiten}.  On the other hand, since $F$ commutes with the Serre functor and translations, it commutes with all powers of $\tau$, so we have
\begin{eqnarray*}
\mathcal{O}(j) & \cong & F(\mathcal{S}) \\
& \cong & F(\tau \mathcal{S}) \\
& \cong & \tau F(\mathcal{S}) \\
& \cong & \mathcal{O}(j-2),
\end{eqnarray*}
which contradicts \cite[Corollary 3.12]{nyman}.
\newline
{\it Step 2:  We show that if $F(\mathcal{O}(j)) \in {\sf H}_{2}[1]$ for some $j$, then the image of ${\sf H}_{1}[0]$ is in ${\sf H}_{2}[1]$.}  Since $\operatorname{Hom}_{{\sf H}_{1}}(\mathcal{O}(j), \mathcal{S}) \neq 0$ for all simples $\mathcal{S} \in {\sf H}_{1}$ by \cite[Proposition 1.10(b)]{lenzingreiten}, it follows from (\ref{eqn.homrep}) that $F(\mathcal{S}) \in {\sf H}_{2}[1]$ for all simples $\mathcal{S}$, and thus that $F(\mathcal{T}) \in {\sf H}_{2}[1]$ for all finite length $\mathcal{T}$ in ${\sf H}_{1}$.  If $F(\mathcal{O}(l)) \in {\sf H}_{2}[0]$ for some $l$, then
\begin{eqnarray*}
\operatorname{Hom}_{D^{b}({\sf H}_{2})}(F(\mathcal{O}(l)),F(\mathcal{S})) & \cong & \operatorname{Ext}_{{\sf H}_{2}}^{1}(F(\mathcal{O}(l)),F(\mathcal{S})) \\
& \cong & D\operatorname{Hom}_{{\sf H}_{2}}(F(\mathcal{S}),F(\mathcal{O}(l-2))) \\
& = & 0
\end{eqnarray*}
where the second isomorphism follows from Serre duality and from the fact that $F$ commutes with translations and the Serre functor.  This contradicts the fact that $\operatorname{Hom}_{{\sf H}_{1}}(\mathcal{O}(l), \mathcal{S}) \neq 0$, and Step 2 follows.
\newline
{\it Step 3:  We show that if $F(\mathcal{O}(j)) \in {\sf H}_{2}[0]$ for all $j$ then the image of ${\sf H}_{1}[0]$ is in ${\sf H}_{2}[0]$.}  First, suppose $F(\mathcal{O}(l))$ is a line bundle for some $l$. If $F(\mathcal{S}) \in {\sf H}_{2}[1]$ for some simple $\mathcal{S} \in {\sf H}_{1}$, then $0 \neq \operatorname{Hom}_{{\sf H}_{1}}(\mathcal{O}(l), \mathcal{S})$.  On the other hand,
$$
\operatorname{Hom}_{D^{b}({\sf H}_{2})}(F(\mathcal{O}(l)),F(\mathcal{S})) \cong \operatorname{Ext}^{1}_{{\sf H}_{2}}(F(\mathcal{O}(l)), F(\mathcal{S})) =0
$$
where the last equality follows from Serre duality and Step 1.  Step 3 follows in this case.  Next, suppose $F(\mathcal{O}(j))$ has finite length for all $j$.  Then, since $F$ commutes with translation, Step 1 implies that the image of $F$ contains no line bundles, which contradicts the fact that $F$ is an equivalence.

The second part of the theorem follows from the first and Theorem \ref{thm.auto}.
\end{proof}

\subsection{Applications to homogeneous noncommutative curves of genus zero}
We now describe our contributions to Questions \ref{question.first}, \ref{question.third}, and \ref{question.second}.  For the remainder of this section, we assume ${\sf H}$ is a homogeneous noncommutative curve of genus zero with underlying $(2,2)$-bimodule $M$ and with $\operatorname{End}(\mathcal{L})=\operatorname{End}(\overline{\mathcal{L}})$ commutative.  We also assume $k$ is perfect and $\operatorname{char }k \neq 2$ so that we may invoke Theorem \ref{thm.auto}.  Thus, by Corollary \ref{cor.main}, we may identify ${\sf H}$ with $\mathbb{P}^{n.c.}(M)$.

We will also use the following notation: if ${\sf Gr }\mathbb{S}^{n.c.}(M)$ is the category of graded right $\mathbb{S}^{n.c.}(M)$-modules, then we let $\pi: {\sf Gr }\mathbb{S}^{n.c.}(M) \longrightarrow \mathbb{P}^{n.c.}(M)$ denote the quotient functor.

\subsubsection{Question \ref{question.first}}
Our contribution towards Question \ref{question.first} is the following:
\begin{prop} \label{prop.first}
The space ${\sf H}$ has a $k$-linear autoequivalence sending $\mathcal{L}$ to $\overline{\mathcal{L}}$ if and only if $M \cong K_{\sigma} \otimes {}^{*}M \otimes K_{\epsilon}$ for some $\sigma, \epsilon \in \operatorname{Gal}(K/k)$.  This is the case if and only if the group of $k$-linear autoequivalences acts transitively on the set of line bundles in ${\sf H}$
\end{prop}

\begin{proof}
By \cite[Lemma 3.16 and Corollary 3.17]{nyman}, we must have $\mathcal{L} = \pi e_{j}\mathbb{S}^{n.c.}(M)$ for some integer $j$.  Since $M$ is a $(2,2)$-bimodule, it follows that $\overline{\mathcal{L}} = \pi e_{j-1}\mathbb{S}^{n.c.}(M)$.

Now suppose $F$ is a $k$-linear equivalence satisfying the hypothesis.  By Theorem \ref{thm.auto}, $F$ is a composition
$$
\mathbb{P}^{n.c.}(M) \overset{T_{\delta,\epsilon}}{\longrightarrow} \mathbb{P}^{n.c.}(K_{\delta^{-1}} \otimes M \otimes K_{\epsilon}) \overset{\Phi}{\longrightarrow} \mathbb{P}^{n.c.}(M^{i*}) \overset{[-i]}{\longrightarrow} \mathbb{P}^{n.c.}(M).
$$
Since $F(\mathcal{L}) \cong \overline{\mathcal{L}}$, it follows from \cite[Lemma 4.1, Lemma 4.2(1) and Lemma 4.8(1)]{nyman} that $i=-1$.  Therefore, $\Phi$ is induced by an isomorphism $K_{\delta^{-1}} \otimes M \otimes K_{\epsilon} \longrightarrow {}^{*}M$ as desired.

Conversely, if there is an isomorphism $\phi: K_{\delta^{-1}} \otimes M \otimes K_{\epsilon} \longrightarrow {}^{*}M$, then the equivalence
$$
\mathbb{P}^{n.c.}(M) \overset{T_{\delta,\epsilon}}{\longrightarrow} \mathbb{P}^{n.c.}(K_{\delta^{-1}} \otimes M \otimes K_{\epsilon}) \overset{\Phi}{\longrightarrow} \mathbb{P}^{n.c.}({}^{*}M) \overset{[1]}{\longrightarrow} \mathbb{P}^{n.c.}(M)
$$
such that $\Phi$ is induced by $\phi$ sends $\mathcal{L}$ to $\overline{\mathcal{L}}$ by \cite[Lemma 4.1, Lemma 4.2(1) and Lemma 4.8(1)]{nyman}.

The last result follows from the fact that every line bundle in ${\sf H}$ is either of the form $\tau^{i}\mathcal{L}$ for some $i$ or of the form $\tau^{i}\overline{\mathcal{L}}$ for some $i$.
\end{proof}

\subsubsection{Question \ref{question.third}}
In order to address Question \ref{question.third}, we need the following result.
\begin{lemma}
The automorphism group $\operatorname{Aut }(\mathbb{X})$ of ${\sf H}$ (defined in Section \ref{section.questions}) is isomorphic to the group of shift-free $k$-linear autoequivalences of $\mathbb{P}^{n.c.}(M)$.
\end{lemma}

\begin{proof}
If $F$ is a $k$-linear autoequivalence of $\mathbb{P}^{n.c.}(M)$, then by Theorem \ref{thm.auto} (and in the notation introduced before, and in, the statement of the theorem),  $F \cong [-i] \circ \Phi  \circ T_{\delta, \epsilon}$.  Since $\mathcal{L}$ is of the form $\pi e_{j}\mathbb{S}^{n.c.}(M)$, it follows from \cite[Lemma 4.1, Lemma 4.2(1) and Lemma 4.8(1)]{nyman} that $F(\mathcal{L}) \cong \pi e_{j+i}\mathbb{S}^{n.c.}(M)$.  Thus, $F(\mathcal{L}) \cong \mathcal{L}$ if and only if $i=0$, i.e. if and only if $F$ is shift-free.
\end{proof}

The group of shift-free autoequivalences of ${\mathbb{P}}^{n.c.}(M)$, $\operatorname{Aut }\mathbb{P}^{n.c.}(M)$, is described in \cite{nyman}.  For the readers convenience, we recount the result here.  In order to proceed, we describe some notation:  let $\operatorname{Stab}M$ denote the subgroup of $\mbox{Gal }(K/k)^{2}$ consisting of $(\delta, \epsilon)$ such that $K_{\delta^{-1}} \otimes_{K} M \otimes_{K} K_{\epsilon} \cong M$ and let $\operatorname{Aut }M$ denote the group (under composition) of bimodule isomorphisms $M \longrightarrow M$ modulo the relation defined by setting $\phi' \equiv \phi$ if and only if there exist nonzero $a, b \in K$ such that $\phi'  \phi^{-1}(m)=a \cdot m \cdot b$ for all $m \in M$.

We will also need to recall some preliminaries regarding simple bimodules.  In \cite{papp}, simple left finite dimensional $k$-central $K-K$-bimodules are classified.  Since this classification will be invoked in what follows, we recall it now.  Let $\overline{K}$ be an algebraic closure of $K$, let $\operatorname{Emb}(K)$ denote the set of $k$-algebra maps $K \longrightarrow \overline{K}$, and let $G=\operatorname{Gal}(\overline{K}/K)$.  Now, $G$ acts on $\operatorname{Emb}(K)$ by left composition. Given $\lambda\in \operatorname{Emb}(K)$, we denote the orbit of $\lambda$ under this action by $\lambda^G$.  We denote the set of finite orbits of $\operatorname{Emb}(K)$ under the action
of $G$ by $\Lambda(K)$.

\begin{thm} \label{thm.papp} \cite{papp} There is a one-to-one correspondence
between $\Lambda(K)$ and isomorphism classes of simple left finite dimensional
$k$-central $K-K$-bimodules.
\end{thm}
In fact, the class of the simple bimodule corresponding to $\lambda$ is the class of the bimodule, $V(\lambda)$, defined as follows: as a set, $V(\lambda)=K \vee \lambda(K)$, and this set with its usual field structure is denoted $K(\lambda)$.  The left and right action are via multiplication in this field: $x \cdot v := xv$ while $v \cdot x := v\lambda(x)$ (\cite[Proposition 2.3]{hart}).

The group $\operatorname{Aut }\mathbb{P}^{n.c.}(M)$ is described in the following
\begin{thm}
There is a group homomorphism
$$
\theta: \operatorname{Stab}(M)^{op} \longrightarrow \operatorname{Aut}(\operatorname{Aut}(M))
$$
(described below) such that
$$
\operatorname{Aut }\mathbb{P}^{n.c.}(M) \cong \operatorname{Aut }M \rtimes_{\theta} \operatorname{Stab}M^{op}.
$$
Furthermore, the factors of $\operatorname{Aut}\mathbb{P}^{n.c.}(M)$ are described as follows:
\begin{enumerate}

\item{} If $M \cong K_{\sigma} \oplus K_{\epsilon}$, where $\sigma, \epsilon \in \operatorname{Gal}(K/k)$ and $\sigma \neq \epsilon$, then
$$
\operatorname{Aut}M \cong K^{*} \times K^{*}/\{(a\sigma(b),a\epsilon(b))|a, b \in K^{*}\}.
$$
and
$$
\operatorname{Stab}M = \{(\delta,\gamma)|\{\delta^{-1} \sigma \gamma, \delta^{-1} \epsilon \gamma \}=\{\sigma, \epsilon \}\},
$$
\item{} If $M \cong V(\lambda)$, and if, for $\delta \in \operatorname{Gal }(K/k)$, $\overline{\delta}$ denotes an extension of $\delta$ to $\overline{K}$, then
$$
\operatorname{Aut }(M) \cong K(\lambda)^{*}/K^{*}\lambda(K)^{*}.
$$
and
$$
\operatorname{Stab}(M) = \{(\delta, \gamma) \in \operatorname{Gal}(K/k)^{2} | (\overline{\delta}^{-1}\lambda \gamma)^{G}=\lambda^{G}\}.
$$
\end{enumerate}
\end{thm}

To describe $\theta$, we will need to recall \cite[Lemma 2.2]{nyman}:

\begin{lemma} \label{lemma.autocomp}
Suppose $\sigma, \epsilon \in \operatorname{Gal}(K/k)$ and let $\overline{\sigma}$ denote an extension of $\sigma$ to $\overline{K}$.  Under the isomorphism given in Theorem \ref{thm.papp}, the $G$-orbit of $\overline{\sigma} \lambda \epsilon$ corresponds to the isomorphism class of the simple bimodule $K_{\sigma} \otimes V(\lambda) \otimes K_{\epsilon}$.
\end{lemma}

According to \cite[Lemma 2.1]{nyman}, there are only two possibilities for $M$ such that $\mathbb{P}^{n.c.}(M)$ is not equivalent to the commutative projective line.  We define $\theta$ according to the corresponding structure of $M$.

\begin{enumerate}

\item{} In case $\sigma, \epsilon \in \operatorname{Gal}(K/k)$, $\sigma \neq \epsilon$, $M=K_{\sigma} \oplus K_{\epsilon}$, and $(\delta, \gamma) \in \operatorname{Stab}M$ is such that $\delta^{-1} \sigma \gamma = \sigma$, we define $\theta(\delta, \gamma)$ as the element in
    $$
    \operatorname{Aut}(K^{*} \times K^{*}/\{(a \sigma(b),a \epsilon(b))|a,b \in K^{*}\})
    $$
    which sends the class of the pair $(c, d) \in K^{*} \times K^{*}$ to the class of the pair $(\delta^{-1}(c), \delta^{-1}(d))$.  In other words, $\theta(\delta, \gamma)$ acting on the class of an isomorphism $\phi$ equals the class of $\delta^{-1} \phi \delta$, where $\delta^{-1}$ and $\delta$ act coordinate-wise on $K^{2}$.

    If $\delta^{-1} \sigma \gamma = \epsilon$, then we define $\theta(\delta, \gamma)$ as automorphism sending the class of $(c, d)$ to the class of $(\delta^{-1}(d), \delta^{-1}(c))$.  In other words, $\theta(\delta, \gamma)$ acting on the class of an isomorphism $\phi$ equals the class of $\upsilon \delta^{-1} \phi \delta \upsilon$, where $\delta^{-1}$ and $\delta$ act coordinate-wise on $K^{2}$ and $\upsilon$ is the linear transformation that exchanges the coordinates.

\item{} In case $M=V(\lambda)$ and $(\delta, \gamma) \in \operatorname{Stab}M$, we define $\theta(\delta, \gamma)$ as the function sending the class of $c \in K(\lambda)^{*}$ to the class of $\psi^{-1}(c)$, where $\psi:K(\lambda) \longrightarrow K(\lambda)$ is the $k$-algebra isomorphism defined in \cite[Proposition 7.1]{nyman}.  Specifically, there exists some extension of $\delta$, $\overline{\delta}$, such that ${\overline{\delta}}^{-1}\lambda \gamma = \lambda$ by Lemma \ref{lemma.autocomp}.  We let $\psi=\overline{\delta}|_{K(\lambda)}$.

    Thus, $\theta(\delta, \gamma)$ acting on the class of an isomorphism $\phi$ equals the class of $\psi^{-1} \phi \psi$.
\end{enumerate}

\subsubsection{Question \ref{question.second}}
We now describe the Auslander-Reiten translation functor of ${\sf H}$.  As described in \cite{nyman}, it is simply the equivalence
$$
\mathbb{P}^{n.c.}(M) \overset{\Phi}{\longrightarrow} \mathbb{P}^{n.c.}(M^{**}) \overset{[-2]}{\longrightarrow} \mathbb{P}^{n.c.}(M)
$$
where $\Phi$ is induced by an isomorphism $\phi:M \longrightarrow M^{**}$ (which exists, for example, by \cite[Theorem 3.13]{hart}).  This address Question \ref{question.second}.

\end{document}